\documentclass{amsart}
\usepackage{hyperref}
\usepackage{graphicx}
\usepackage{curves}
\usepackage{tikz}
\usetikzlibrary{arrows}
\usepackage{enumerate}
\usepackage{float}
\newcommand{\arxiv}[1]{\href{http://arxiv.org/abs/#1}{arXiv:#1}}
\newcommand*{\mailto}[1]{\href{mailto:#1}{\nolinkurl{#1}}}

\newtheorem{theorem}{Theorem}[section]
\newtheorem{lemma}[theorem]{Lemma}
\newtheorem{corollary}[theorem]{Corollary}

\newcommand{\R}{\mathbb{R}}
\newcommand{\Z}{\mathbb{Z}}

\newcommand{\C}{\mathbb{C}}

\newcommand{\beq}{\begin{equation}}
\newcommand{\eeq}{\end{equation}}
\newcommand{\bea}{\begin{eqnarray}}
\newcommand{\eea}{\end{eqnarray}}

\newcommand{\ti}{\tilde}

\newcommand{\I}{\mathrm{i}}
\newcommand{\E}{\mathrm{e}}

\newcommand{\im}{\mathop{\mathrm{Im}}}

\newcommand{\FOm}{\mathcal{F}(\ti{\Lambda})}

\newcommand{\si}{\sigma}

\newcommand{\ga}{\gamma}
\newcommand{\om}{\omega}

\numberwithin{equation}{section}

\newcommand{\sigI}{\begin{pmatrix} 0 & 1 \\ 1 & 0 \end{pmatrix}}

 \newcommand{\noprint}[1]{}

\begin{document}

\title[Riemann--Hilbert problem on the torus]{A scalar Riemann--Hilbert problem on the torus:
\\
Applications to the KdV equation}

\author[M. Piorkowski]{Mateusz Piorkowski}
\address{Faculty of Mathematics\\ University of Vienna\\
Oskar-Morgenstern-Platz 1\\ 1090 Wien}
\email{\href{mailto:Mateusz.Piorkowski@univie.ac.at}{Mateusz.Piorkowski@univie.ac.at}}

\author[G. Teschl]{Gerald Teschl}
\address{Faculty of Mathematics\\ University of Vienna\\
Oskar-Morgenstern-Platz 1\\ 1090 Wien}
\email{\href{mailto:Gerald.Teschl@univie.ac.at}{Gerald.Teschl@univie.ac.at}}
\urladdr{\href{http://www.mat.univie.ac.at/~gerald/}{http://www.mat.univie.ac.at/\string~gerald/}}

\thanks{Research supported by the Austrian Science Fund (FWF) under Grants No.\ P31651 and W1245.}

\begin{abstract}
We take a closer look at the Riemann--Hilbert problem associated to one-gap solutions of the Korteweg--de Vries equation. To gain more insight, we reformulate it as a scalar Riemann--Hilbert problem on the torus. This enables us to derive deductively the model vector-valued and singular matrix-valued solutions in terms of Jacobi theta functions. We compare our results with those obtained in recent literature.
\end{abstract}

\maketitle

\section{Introduction}
\subsection{Background} The main goal of this short note is to present an alternative approach to the existence/uniqueness results for  the model Riemann--Hilbert (R-H) problem presented in \cite{EGKT} and the construction of a singular matrix-valued solution found in \cite[Sect.~6]{EPT} (see also \cite[Sect.~3]{GGM}). Recall, that the objective of \cite{EGKT} and \cite{EPT} was to apply rigorously the nonlinear steepest descent method to the initial value problem for the Korteweg--de Vries (KdV) equation,
\beq\nonumber
q_t(x,t)=6q(x,t)q_x(x,t)-q_{xxx}(x,t), \quad (x,t)\in\R\times\R_+,
\eeq
with steplike initial data $q(x,0) = q(x)$:
\beq\nonumber \lim_{x \to \infty} q(x) = 0, \qquad \lim_{x \to -\infty} q(x) = -c^2, \quad c > 0.
\eeq
For large $t$, solutions to this problem display  different behaviours in three regions of the $(x,t)$-plane characterized by the ratio $x/t$ (see \cite[Sect.~1]{EGKT}). Of particular interest to us is the transition region given by $-6c^2 t < x < 4c^2 t$, where solutions asymptotically converge to a modulated elliptic wave. This result was proven in \cite{EPT}, where an ill-posedness of the corresponding holomorphic matrix model R-H problem was found.

The ill-posedness is closely related to the fact that the R-H problem for the KdV equation is formulated as a vector-valued problem. Note that, the standard Liouville-type argument relating existence to uniqueness for matrix-valued R-H problems having jump matrices with unit determinant (see for example \cite[Thm.~5.6]{JL}), cannot be generalized to the vector case in a straightforward manner. In fact, uniqueness can fail despite existence, as demonstrated in \cite[Sect.~2]{GT} for the simple case of a one soliton solution. Uniqueness was restored by assuming an additional symmetry condition.  

Another feature of the KdV equation playing an important role in the present note is the relationship between finite-gap solutions and elliptic Riemann surfaces (see \cite[Ch.~3]{BBEIM}). Algebro-geometric finite-gap solutions to the KdV equation can be given explicitly in terms of Jacobi theta functions via the Its--Matveev formula \cite{IM} (see also \cite{EGT}). Unsurprisingly, the solution of the corresponding model R-H problem is also expressed in terms of Jacobi theta functions. Given that these functions can be regarded as multivalued functions on an underlying Riemann surface, the natural question arises whether the model R-H problem in the plane found in \cite{EGKT} can be viewed as a R-H problem on a Riemann surface instead. In our simple one-gap case, that would correspond to a R-H problem on a torus. 

\subsection{Outline of this work}

In the next section we will show that the one-gap KdV model R-H problem can in fact be formulated as a \emph{scalar-valued} R-H problem on the torus. Equivalently, solutions to this problem can be characterized by quasiperiodic meromorphic functions in the complex plane (see Eq.~\eqref{eqjcav}), leading to the explicit R-H model solution found in \cite{EGKT} and singular solutions similar to the one described in \cite{EPT} (see also \cite{GGM}) in a straightforward manner.  Moreover, we show that the symmetry condition from \cite[Sect.~2]{GT} translates to halving the period (see Eq.~\eqref{SymCond}), while uniqueness follows from Liouville's Theorem.

Section~\ref{sec:mv} compares different regular and singular matrix-valued model solutions. As shown in \cite{EPT}, there is no regular matrix-valued model solution satisfying all the standard assumptions, hence it is necessary to drop some of them. We also comment on the regularity of the determinant of the solution in each case.

In the final section we compare our singular matrix-valued model solution to the ones found previously (see \cite{EPT}, \cite{GGM}). We point out that the corresponding vanishing problem has a nontrivial solution, meaning that there is no uniqueness for the associated singular model problem. In particular, the solutions described in \cite{EPT} and \cite{GGM} differ from the one we presented in Section~\ref{sec:mv}.

\section{The model Riemann--Hilbert problem}
\label{ModelRH}

In the following we recall the model vector-valued R-H problem for one-gap solutions of the KdV equation. For the underlying scattering theory and nonlinear steepest descent analysis leading to this problem in the transition region, we refer to \cite[Sect.~4]{EGKT}.

Find a vector-valued  function $m^{\text{mod}}(k)=(m_1^{\text{mod}}(k),\ m_2^{\text{mod}}(k))$ holomorphic in the domain $\C\setminus [\I c,  -\I c]$, continuous up to the boundary except at points $\mathcal G:=\{ \I c, \I a, -\I a, -\I c\}$ and satisfying the jump condition (with $\ti{\Lambda}=\frac{1}{2\pi}(\Lambda + t B)$, cf.~\cite[Sect.~3]{EPT}):
\beq\label{defmvecmod}
m_+^{\text{mod}}(k)= m_-^{\text{mod}}(k) v^{\text{mod}}(k),
\eeq
where
\beq\label{jumpcondmod}
v^{\text{mod}}(k) = \left\{ \begin{array}{ll}
\begin{pmatrix}
0 & \I \\
\I  & 0
\end{pmatrix},& k\in [\I c, \I a], \,\\
\begin{pmatrix}
0 & -\I \\
-\I & 0
\end{pmatrix},& k\in [-\I a, -\I c],\\
\begin{pmatrix}\E^{-2\pi\I\ti{\Lambda}}& 0\\
0&\E^{2\pi\I\ti{\Lambda}}\end{pmatrix},& k\in [\I a, -\I a],\\
\end{array}\right.
\eeq
the symmetry condition
\beq \label{symcond}
m^{\text{mod}}(-k) = m^{\text{mod}}(k) \sigI,
\eeq
and the normalization condition
\beq\label{normcond}
\lim_{k\to\I\infty} m^{\text{mod}}(k)= (1\ \ 1).\eeq
At any point $\kappa\in\mathcal G$ the vector function $m^{\text{mod}}(k)$ can have at most a fourth root singularity:  $m^{\text{mod}}(k)= O((k-\kappa)^{-1/4})$, $k\to \kappa$.
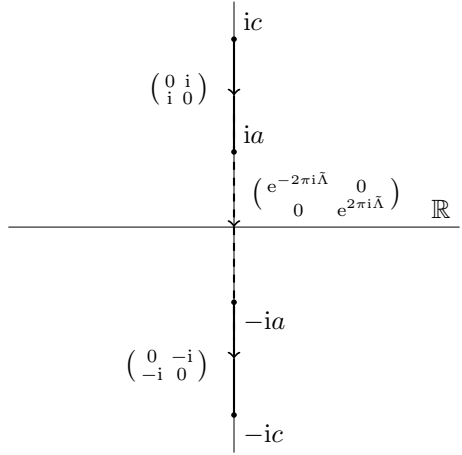
\begin{figure}[H]
\begin{tikzpicture}
\draw[very thin] (-3,0) -- (3,0);
\draw[very thin] (0,-3) -- (0,3);
\draw[thick,->] (0,-1) -- (0,-1.75);\draw[thick] (0,-1.75) -- (0,-2.5);
\draw[thick,->] (0,2.5) -- (0,1.75);\draw[thick] (0,1.75) -- (0,1);
\draw[thick,dashed,->] (0,1.5) -- (0,0);\draw[thick,dashed] (0,0) -- (0,-1.5);
\draw[fill] (0,2.5) circle (0.03) node[above right] {$\I c$};
\draw[fill] (0,-2.5) circle (0.03) node[below right] {$-\I c$};
\draw[fill] (0,1) circle (0.03) node[above right] {$\I a$};
\draw[fill] (0,-1) circle (0.03) node[below right] {$-\I a$};

\node[above right] at (2.5,0) {$\mathbb{R}$};

\draw (0.2,0.2) node[above right,xshift=-3,yshift=-5] {$\big(\begin{smallmatrix} \E^{-2\pi\I\ti{\Lambda}}& 0\\ 0 & \E^{2\pi\I\ti{\Lambda}}\end{smallmatrix}\big)$};
\draw (-0.2,-1.5) node[below left] {$\big(\begin{smallmatrix} 0& -\I\\ -\I & 0\end{smallmatrix}\big)$};
\draw (-0.2,1.5) node[above left] {$\big(\begin{smallmatrix} 0& \I\\ \I & 0\end{smallmatrix}\big)$};

\end{tikzpicture}
\caption{Jump contour for the model R-H problem}
\end{figure}
The solution to this problem was given in \cite{EGKT}. As mentioned in the introduction, we want to solve this problem in a slightly different way, which should
shed some further light on the model problem. For this, it will be convenient to denote by $m(k)$ a generic vector-valued meromorphic function satisfying the jump condition \eqref{jumpcondmod}.

For our first transformation we define
\beq\nonumber
\ti{\ga}(k) = \sqrt[4]{\frac{k^2+a^2}{k^2+c^2}}
\eeq
with the branch cuts along $[\pm\I a, \pm\I c]$ and the branch chosen such that $\ti{\ga}(k)>0$ for $k\in [\I c,\infty)$.
Note that we have $\ti{\ga}(-k)=\ti{\ga}(k)$ and $\ti{\ga}(k)>0$ for $k\in\R$.
Then $\ti{\ga}(k)$ solves the scalar R-H problem
\[
\ti{\ga}_+(k) = \pm\I \ti{\ga}_-(k), \qquad k\in[\pm\I a, \pm\I c],
\]
and we set
\beq \label{mn}
m(k) = \ti{\ga}(k) n(k)
\eeq
such that $n(k)$ satisfies the jump condition \eqref{jumpcondmod}, except that the jumps on $[-\I c,-\I a]$ and $[\I c,\I a]$ are replaced by
\beq\label{eqjumpsigI}
n_+(k) = n_-(k) \sigI.
\eeq
The reason for this change is that it will be convenient to look at this problem on the elliptic Riemann surface $X$ associated with the function
\[
w(k)=\sqrt{(k^2 + c^2)(k^2 + a^2)},
\]
defined on $\C \setminus ([-\I c, -\I a] \cup [\I a, \I c])$ with $w(0) > 0$. The two sheets of $X$ are glued along the cuts  $[\I c,\I a]$ and $[-\I a, -\I c]$. Points on this surface are denoted by $p=(k,\pm )$. To simplify formulas we keep the notation  $k=(k,+)$ for points on the upper sheet of $X$.

In this setup, the two components $n_1$, $n_2$ of the vector $n \colon \C\setminus[-\I c,\I c] \to \C^2$ can be regarded as the values of a single function $N \colon X \to \C$ on the upper, lower sheet, respectively. Explicitly,
\beq \label{nN}
n(k) = (N((k,+)),N((k,-))).
\eeq
In this case the jump condition \eqref{eqjumpsigI} implies that $N$ will have no jump along the cuts, where the two sheets are glued together. However, the other jump
will remain. In fact, the jump contour on $X$ is a circle through the two branch points $-\I a$ and $\I a$, on which we have the jump condition
\beq\label{jumpN}
N_+(p) = N_-(p) \E^{-2\pi\I\ti{\Lambda}}.
\eeq
Note that the symmetry condition \eqref{symcond} translates to
\beq\label{symN}
N(p^*) = N(-p),
\eeq
where $(k,\pm)^* = (k,\mp)$ denotes the sheet exchange map and $-(k,\pm)=(-k,\pm)$.

Next we choose a canonical homology basis of cycles $\{\bf a, \bf b\}$ as follows: The $\bf a$-cycle surrounds the points $-\I a,\I a$ starting on the upper sheet from the left side of
the cut $[\I c,\I a]$ and continues on the upper sheet to the left part of $[-\I a, -\I c]$ and returns after changing sheets. The cycle $\bf b$ surrounds the points $\I a, \I c$ counterclockwise on the upper sheet.

Then the normalized holomorphic differential is given by
\beq\nonumber
d\om=\Gamma\frac{d\zeta}{w(\zeta)}, \quad \text{where}\ \Gamma:=\left(\int_{\bf a} \frac{d \zeta}{w(\zeta)}\right)^{-1}\in \I \R_-,
\eeq
such that $\int_{\bf a}d\om=1$ and 
\beq\nonumber
\tau=\int_{\bf b} d\om\in \I \R_+.
\eeq
Let
\beq\nonumber
\theta_3(z\,\big|\,\tau)=\sum_{n\in\Z}\exp\big( (n^2\tau + 2n z)\pi\I \big), \quad z\in\C,
\eeq
be the associated Jacobi theta function (see for example \cite{Byrd}). Recall that $\theta_3$ is even, $\theta_3(-z\,\big|\,\tau)=\theta_3(z\,\big|\,\tau)$, and satisfies
\beq\label{jthper}
\theta_3(z+ n + \tau\ell\,\big|\,\tau)=\theta_3(z\,\big|\,\tau )\E^{- \pi\I \tau \ell^2 - 2\pi\I\ell z} \quad\text{for}\quad \ell,n \in \Z.
\eeq
Furthermore, let $A(p)=\int_{\I c}^p d\om$ be the Abel map on $X$. We identify the upper sheet of $X$ with the complex plane $\C\setminus([\I c, \I a]\cup [-\I a, -\I c])$.
Restricting the path of integration to $\C \setminus [\I c, -\I c]$ we observe that $A(k)$ is a holomorphic function in that given domain with the following properties:
\begin{itemize}
\item $A_+(k)=-A_-(k) \pmod{1}$, \ for $k\in [\I c, \I a]\cup [-\I a, -\I c]$;
\item $A_+(k)-A_-(k)=-\tau$, \ for $k\in [\I a, -\I a]$;
\item $A(-k)=-A(k) + \frac{1}{2}$, \ for $k\in\C\setminus [\I c, -\I c]$, 
\item $A_+(\I a)=-\frac{\tau}{2} =-A_-(\I a)$, $A_+(-\I a) = -\frac{\tau}{2} + \frac{1}{2}$,  $A_-(-\I a)=\frac{\tau}{2} + \frac{1}{2}$.
\item $A(\infty)=\frac{1}{4}$, $A(k)= \frac{1}{4} -\Gamma k^{-1} + O(k^{-3})$, \ as $k\to\infty$.
\end{itemize}
For points on the lower sheet we set $A(p^*)=-A(p)$.
Finally, denote by $K= \frac{1+\tau}{2}$ the Riemann constant associated with $X$ and abbreviate $\infty_\pm= (\infty,\pm)$, $0_\pm= (0,\pm)$. Note that
$A(0_+)=\frac{1}{4}+\frac{\tau}{2}$. By Riemann's vanishing theorem \cite{FK} the zeros of $\theta_3$ are simple and given by $z = K +\Z+\tau \Z$.

According to the Jacobi inversion theorem \cite{FK}, the Abel map $A$ maps our Riemann surface $X$ bijectively to its associated Jacobi variety $\C/ (\Z+\tau\Z)$ depicted in Figure~\ref{figjv}.
\begin{figure}[ht]
\begin{tikzpicture}
\path [fill=gray!40] plot (0,-1.5) -- (2,-1.5) -- (2,1.5) -- (0,1.5) --cycle;
\path [fill=gray!20] plot (-2,-1.5) -- (0,-1.5) -- (0,1.5) -- (-2,1.5) --cycle;
\draw (-4.5,3) -- (4.5,3);
\draw (-4.5,0) -- (4.5,0);
\draw (-4.5,-3) -- (4.5,-3);
\draw (-4,-3.5) -- (-4,3.5);
\draw (0,-3.5) -- (0,3.5);
\draw (4,-3.5) -- (4,3.5);
\draw[dashed] (-4.5,1.5) -- (4.5,1.5);
\draw[dashed] (-4.5,-1.5) -- (4.5,-1.5);

\draw (4,0) node[above right] {$1$};
\draw (0,3) node[above right] {$\tau$};

\draw[fill] (0,0) circle (0.03) node[above right] {$\I c$};
\draw[fill] (-2,0) circle (0.03) node[above left] {$-\I c$};
\draw[fill] (2,0) circle (0.03) node[above right] {$-\I c$};
\draw[fill] (-1,0) circle (0.03) node[below] {$\infty_-$};
\draw[fill] (1,1.5) circle (0.03) node[below] {$0_+$};
\draw[fill] (-1,1.5) circle (0.03) node[below] {$0_-$};
\draw[fill] (1,-1.5) circle (0.03) node[below] {$0_+$};
\draw[fill] (-1,-1.5) circle (0.03) node[below] {$0_-$};
\draw[fill] (1,0) circle (0.03) node[below] {$\infty_+$};
\draw[fill] (0,1.5) circle (0.03) node[above right] {$\I a$};
\draw[fill] (0,-1.5) circle (0.03) node[below right] {$\I a$};
\draw[fill] (-2,1.5) circle (0.03) node[above] {$-\I a$};
\draw[fill] (2,1.5) circle (0.03) node[above] {$-\I a$};
\draw[fill] (-2,-1.5) circle (0.03) node[below] {$-\I a$};
\draw[fill] (2,-1.5) circle (0.03) node[below] {$-\I a$};

\end{tikzpicture}
\caption{Jacobi variety (dark/light gray denotes the upper/lower sheet)}\label{figjv}
\end{figure}
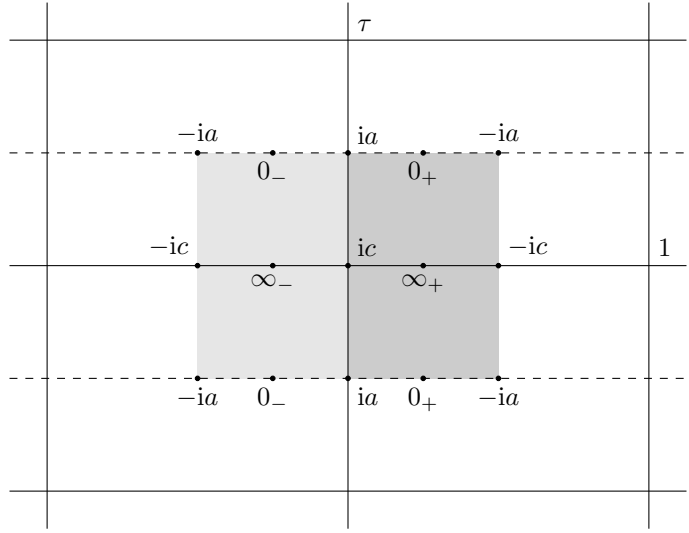
The jump contour is indicated by the dashed line, while the dark/light shaded region correspond to the upper/lower sheet. Moreover, a meromorphic function $E(z)$ given by
\beq \label{NE}
E(A(p))=N(p)
\eeq
will satisfy our original jump condition if and only if
\beq\label{eqjcav}
E(z+1) = E(z), \qquad E(z+\tau) =  E(z)\E^{2\pi\I\ti{\Lambda}},
\eeq
and it will satisfy the symmetry condition if and only if
\beq \label{SymCond}
E(z+\tfrac{1}{2}) = E(z).
\eeq
If this latter condition holds we will call $E$ symmetric. If we have $E(z+\frac{1}{2}) = -E(z)$, we will call $E$ anti-symmetric. 

At this stage we remind the reader that we have four equivalent ways of describing vector-valued functions satisfying the jump condition \eqref{jumpcondmod}:
\begin{align*}
    m(k) \quad \longleftrightarrow \quad n(k) \quad \longleftrightarrow \quad N(p) \quad \longleftrightarrow \quad E(z)
\end{align*}
related via \eqref{mn}, \eqref{nN} and \eqref{NE} respectively.
The most convenient framework will be given through the quasiperiodic meromorphic functions $E(z)$. Let us consider the space $\FOm$ of all quasiperiodic meromorphic functions on $\C/ (\Z+\tau\Z)$ satisfying \eqref{eqjcav}, without imposing any symmetry requirements.
Note that $E\in\FOm$ is uniquely determined up to a constant by its divisor $(E)= \sum_{j=1}^n \mathcal{D}_{z_j}- \sum_{j=1}^n \mathcal{D}_{p_j}$, 
since the quotient of two such functions with the same divisor is elliptic without poles, hence a constant. Moreover, since $\frac{E'}{E}$ is elliptic, integrating this function along a fundamental polygon
shows that the number of zeros and poles must be equal. Note also that there must be at least one pole (unless $\ti{\Lambda}=0$). Integrating
$z\frac{E'(z)}{E(z)}$ along a fundamental polygon gives
\beq\label{eqabelom}
\sum_{j=1}^n z_j - \sum_{j=1}^n p_j = \ti{\Lambda} \pmod{\Z+\tau\Z}.
\eeq
Choosing representatives $z_j$, $p_j \in \C$ such that
\beq\label{eqabelom2}
\sum_{j=1}^n z_j - \sum_{j=1}^n p_j = \ti{\Lambda} \pmod{\Z},
\eeq
we can represent $E$ as
\beq\label{eqFtheta}
E(z) = E_0\prod_{j=1}^n \frac{\theta_3(z-z_j-K|\tau)}{\theta_3(z-p_j-K|\tau)}.
\eeq
Indeed the right-hand side has the required zeros and poles while \eqref{jthper} and \eqref{eqabelom2} ensure that it is elliptic.

\begin{lemma}
The divisor of $E\in\FOm$ is invariant with respect to translations of $\frac{1}{2}$ if and only if $E$ is either symmetric or anti-symmetric
\end{lemma}

\begin{proof}
Observe that $C=\frac{E(z+\frac{1}{2})}{E(z)}$ is elliptic without poles and hence constant. Moreover, $C^2=\frac{E(z+\frac{1}{2})}{E(z)} \cdot \frac{E(z+1)}{E(z+\frac{1}{2})} =1$
shows $C=\pm 1$. The converse is trivial.
\end{proof}

\begin{lemma} \label{uniqueLemma}
If $E\in\FOm$ is (anti-)symmetric and has at most two poles, it is already uniquely determined up to a constant by its poles. Conversely,
for each choice of two poles $p_1$, $p_2=p_1+\frac{1}{2}$ there is a unique (up to constants) symmetric and a unique anti-symmetric
function $E\in\FOm$, with at most simple poles at $p_1$ and $p_2$. In fact $p_1$, $p_2$ are simple poles, unless $\tilde \Lambda \in \mathbb{Z}$, in which case the symmetric solution is constant.  
\end{lemma}

\begin{proof}
Let $E$ be (anti-)symmetric and nonconstant. Denote its poles by $p_1$, $p_2=p_1+\frac{1}{2} \pmod{\Z+\tau\Z}$ and its zeros by $z_1$, $z_2=z_1+\frac{1}{2} \pmod{\Z+\tau\Z}$.
Choosing representatives in $\C$, \eqref{eqabelom} implies $2(z_1- p_1) = \tilde \Lambda + m +n \tau$ for some $m,n\in\Z$.
In particular, since adding a period to $z_1$ is irrelevant, we can assume $m,n\in\{0,1\}$. If $m=1$ this just amounts to exchanging
$z_1$ and $z_2$ and hence we can assume $m=0$ without loss of generality. Now using
$z_1 = p_1 + \frac{\tilde \Lambda}{2} + n\frac{\tau}{2}$, we can set $z_2= z_1 + \frac{1}{2} - n \tau$ and $p_2 = p_1 + \frac{1}{2}$ such that \eqref{eqabelom2}  holds. Now one computes using \eqref{jthper} that \eqref{eqFtheta} fulfills $E(z+\frac{1}{2})=(-1)^n E(z)$. In other words, $p_1 \in \C/(\Z+\tau\Z)$ and $n \in \Z_2$ uniquely determine $E$ up to
a constant. One can check, that the zeros and poles cancel if and only if $n = 0$ and $\tilde \Lambda \in \mathbb{Z}$, corresponding to a constant symmetric solution.
\end{proof}

\begin{corollary}
If $E\in\FOm$ has at most two poles $p_1$, $p_2=p_1+\frac{1}{2}$ then there exist unique $c_s,c_a \in \C$ such that $E=c_s E_s + c_a E_a$,
where $E_s$, $E_a$ are the symmetric, anti-symmetric solutions constructed in the previous lemma, respectively.
\end {corollary}

Returning to our original model problem, we want the poles of $E$ to lie at the images of $\I a$ and $-\I a$ under the Abel map $A$, that is $p_1=\frac{\tau}{2}$ and $p_2=\frac{1+\tau}{2} = K$. The reason is that we require $m^{\text{mod}}(k)$ to be holomorphic, with at most fourth root singularities at points of $\mathcal G$. As $\ti \ga(k)$ has fourth root zeros at $\pm \I a$ and the Abel map $A$ has  square root singularities at the points of $\mathcal G$, simple poles at $\frac{\tau}{2}$, $\frac{1+\tau}{2}$ in $\C/(\Z+\tau\Z)$  translate to fourth root singularities at $\pm \I a$ of $m^{\text{mod}}(k)$ under the inverse of the Abel map. In fact, this is the only choice of the pole structure leading to a holomorphic $m^{\text{mod}}(k)$ with at most fourth root singularities. 

For the zeros of the symmetric ($n = 0$) and the anti-symmetric ($n = 1$) solution we use $z_1 = \frac{\ti{\Lambda}}{2} + \frac{(n+1) \tau}{2}$, $z_2=  \frac{\ti{\Lambda}}{2} + \frac{1-(n-1)\tau}{2}$. 
Denote by
\beq \label{SymSol}
E_s(z)= \frac{\theta_3(z-\frac{\ti{\Lambda}}{2}+\frac{1}{2}|\tau)\theta_3(z-\frac{\ti{\Lambda}}{2}|\tau)}{\theta_3(z+\frac{1}{2}|\tau)\theta_3(z|\tau)}
\eeq
the corresponding symmetric and by
\beq \label{AntiSol}
E_a(z)= \frac{\theta_3(z-\frac{\ti{\Lambda}}{2}-\frac{1+\tau}{2}|\tau)\theta_3(z-\frac{\ti{\Lambda}}{2}+\frac{\tau}{2}|\tau)}{\theta_3(z+\frac{1}{2}|\tau)\theta_3(z|\tau)}
\eeq
the corresponding anti-symmetric solution. Using the identity (cf.~ \cite{dubr} formula (1.4.3))
\[
 \textstyle \theta_3(z|\tau) \theta_3(z+ \frac{1}{2} |\tau) = \theta_3(2z+\frac{1}{2} | 2\tau)\,\theta_3(\frac{1}{2} | 2\tau),
 \]
 (note that the quotient of both sides is a holomorphic elliptic function which equals $1$ at $z=\frac{1}{2}$) we can write the formula for $E_s$ somewhat more compactly as
 \[
 E_s(z) = \frac{\theta_3(2z-\ti{\Lambda}+ \frac{1}{2}|2\tau)}{\theta_3(2z+\frac{1}{2}|2\tau)}.
 \]
So for the symmetric case $n=0$, we have $\I\im(z_j) = \frac{\tau}{2} \pmod{\Z+\tau\Z}$ and both zeros will be on $[-\I a, \I a]$ (see Figure \ref{figjv}).
In the anti-symmetric case $n=1$ we have $\I\im(z_j) =0 \pmod{\Z+\tau\Z}$ and both zeros will be on $(\infty,-\I c] \cup [\I c, \infty]$.
In particular, if $\ti{\Lambda}=\frac{1}{2} \pmod {1}$ the two zeros of the anti-symmetric solution will be at $\infty_\pm$ and we cannot normalize at this point.
Moreover, if $\ti{\Lambda}=0 \pmod {1}$ such that we are looking for elliptic functions, we have $E_s(z)=1$ (i.e.\ zeros and poles coincide) and
the zeros of $E_a(z)$ will be at $z_1=0$ and $z_2=\frac{1}{2}+\tau$.

Hence all solutions of \eqref{jumpN} with poles at most at $\pm\I a$ are given by
\[
N(p) = c_s N_s(p) + c_a N_a(p), \quad N_s(p)=E_s(A(p)), \ N_a(p)=E_a(A(p)), \quad c_s,c_a \in \C,
\]
and we have
\[
N(\infty_\pm)= c_s E_s(\tfrac{1}{4}) \pm c_a E_a(\tfrac{1}{4}), \qquad
N(0_\pm)= c_s E_s(\tfrac{1}{4}+\tfrac{\tau}{2}) \pm c_a E_a(\tfrac{1}{4}+\tfrac{\tau}{2}),
\]
with $E_s(\frac{1}{4})\ne 0$ for all $\ti{\Lambda}\in\R$ and $E_a(\frac{1}{4})\ne 0$ for all $\ti{\Lambda}\neq \frac{1}{2} \pmod{1}$. Moreover, $N$ will satisfy
\eqref{symN} if and only if $c_a=0$.

Note that in the special case $\ti{\Lambda}=0 \pmod {1}$ we have (up to constants):
\[
N_s((k,\pm))=1 \qquad N_a((k,\pm))= \frac{k^2+ c^2}{\pm w(k)}
\]
In the case $\ti{\Lambda}=\frac{1}{2} \pmod {1}$ we have (again up to constants):
\[
N_s((k,\pm))= \frac{k}{\sqrt{k^2+a^2}} \qquad N_a((k,\pm))=  \frac{1}{\sqrt{k^2+a^2}},
\]
where the root has the branch cut along $[-\I a,\I a]$.

Returning to our original problem we have shown:

\begin{lemma}
The function
\beq\nonumber
m^{\emph{mod}}(k) = \frac{\ti{\ga}(k)}{N_s(\infty_+)} \big(N_s((k,+)),N_s((k,-))\big)
\eeq
is the unique vector-valued function which is holomorphic in the domain $\C\setminus [\I c,  -\I c]$, has square integrable boundary values,
and satisfies the jump condition \eqref{defmvecmod}, the symmetry condition \eqref{symcond} and the normalization condition \eqref{normcond}.

Specifically, $m^{\text{mod}}(k)$ is continuous up to the boundary except at points of the set
$\mathcal G:=\{ \I c, \I a, -\I a, -\I c\}$ where it has at most a fourth root singularity:  $m^{\text{mod}}(k)= O((k-\kappa)^{-1/4}))$, $k\to \kappa$.
\end{lemma}

\section{Matrix-valued solutions}
\label{sec:mv}

In the framework of the nonlinear steepest descent analysis one usually needs to construct a matrix-valued R-H solutions which is invertible. This is a necessary step to arrive at a small-norm R-H problem which can be solved via a Neumann series. However, while many integrable wave equations like the modified KdV equation \cite{DZ} or the nonlinear Schr\"odinger equation \cite{DZNLS} have a matrix-valued R-H formulation, this is not the case for the KdV equation. 

Recall that the model matrix-valued R-H problem related to the KdV equation has a jump matrix satisfying (see \cite{EGKT}, \cite{GT})
\begin{align} \label{symcondV}
    v^{\text{mod}}(-k) = \sigma_1 (v^{\text{mod}}(k))^{-1} \sigma_1, \qquad \det v(k) = 1, \quad k \in \Sigma.
\end{align}
For the corresponding holomorphic matrix-valued solution $M^{\text{mod}}(k)$, one would then require
\begin{align} \label{normcondM}
    \lim_{k \to \infty} M^{\text{mod}}(k) = \begin{pmatrix}
    1 & 0
    \\
    0 & 1
    \end{pmatrix}.
\end{align}
Moreover, given holomorphicity of $M^{\text{mod}}(k)$, we can derive from \eqref{symcondV} and \eqref{normcondM}:
\begin{align} \label{detcondM}
    M^{\text{mod}}(-k) = \sigma_1 M^{\text{mod}}(k) \sigma_1, \qquad \det M^{\text{mod}}(k) \equiv 1.
\end{align}
Note that \eqref{detcondM} implies that $M^{\text{mod}}(k)$ must have the form
\begin{align*}
    M^{\text{mod}}(k) = \begin{pmatrix}
    \tilde \alpha(k) & \tilde \beta(-k)
    \\
    \tilde \beta(k) & \tilde \alpha(-k)
    \end{pmatrix}
\end{align*}
and that $(1, 1)M^{\text{mod}}(k)$ will satisfy the symmetry condition \eqref{symcond}. 

Let us now present three different ways of writing down matrix-valued solutions of the model R-H problem corresponding to the KdV equation with steplike initial data in the transition region. Each one will violate some standard assumption described above, as satisfying all of them is in general impossible, see \cite[Rem.~4.1]{EPT}.
\subsection{Partial normalization at infinity}
We start with the function
\[
M^{\text{mod}}_1(k) = \ti{\ga}(k) \begin{pmatrix}\alpha_1(k) & \beta_1(-k)\\ \beta_1(k) & \alpha_1(-k)\end{pmatrix}
\]
where
\begin{align*}
\alpha_1(k)&= \frac{1}{2} \big( N_a(\infty_+) N_s(k) + N_s(\infty_+) N_a(k) \big), 
\\
\beta_1(k)&= \frac{1}{2} \big( N_a(\infty_+) N_s(k) - N_s(\infty_+) N_a(k) \big).
\end{align*}
Note that $M^{\text{mod}}_1(k)$ satisfies the symmetry condition in \eqref{detcondM} and satisfies \emph{partially} the normalization \eqref{normcondM}. In fact we have
\begin{align*}
    \lim_{k\to\infty}M^{\text{mod}}_1(k) = N_s(\infty_+) N_a(\infty_+) \begin{pmatrix} 1 & 0\\ 0 & 1\end{pmatrix}, 
\end{align*}
with
\begin{align*}
    \det M^{\text{mod}}_1(k)= N_s(\infty_+)^2 N_a(\infty_+)^2.
\end{align*}
The problem here is that the prefactor vanishes for $\tilde \Lambda = \frac{1}{2}\pmod 1$ as $N_a(\infty_+) = 0$, hence we cannot enforce the normalization \eqref{normcondM} for all $\tilde \Lambda$. In particular, $M^{\text{mod}}_1(k)$ is not invertible for these values of $\tilde \Lambda$. The relation to $m^{\text{mod}}(k)$ is given through
\begin{align*}
    m^{\text{mod}}(k) = \frac{1}{N_s(\infty_+) N_a(\infty_+)} (1, 1)  M^{\text{mod}}(k).
\end{align*}
where one needs to use the rule of l'H\^ospital for $\tilde \Lambda = \frac{1}{2} \pmod 1$. 

\subsection{Different symmetry condition}
The function
\[
M^{\text{mod}}_2(k) = \ti{\ga}(k) \begin{pmatrix} N_s(k) & N_s(-k)\\ N_a(k) & N_a(-k)\end{pmatrix}
\]
is a matrix-valued solution which satisfies the new symmetry condition $M(-k)=\si_3 M(k) \si_1$ and is nondiagonal at infinity. Moreover, 
\[
\det M^{\text{mod}}_2(k)= -2\ti{\ga}(0) N_s(0_+) N_a(0_+) =  - 2 N_s(\infty_+) N_a(\infty_+),
\]
and the relation to the vector-valued model solution is given by
\[
m^{\text{mod}}(k) = \frac{1}{N_s(\infty_+)} (1, 0)  M^{\text{mod}}_2(k),
\]
which does not require the rule of l'H\^osptial as $N_s(\infty_+) \not = 0$ for all $\tilde \Lambda$.
The advantage of this matrix-valued solution is that its determinant has only first order zeros. 

\subsection{Singularity at the origin}
Finally, we write down a matrix-valued solution $M^{\text{mod}}_3(k)$ with determinant constant equal to $1$. The price we have to pay is a singularity at the origin, making $M^{\text{mod}}_3(k)$ a meromorphic solution. To be precise, we will move the poles of the anti-symmetric solution to $\hat p_{1,2} = \pm \frac{1}{4} + \frac{\tau}{2}$, which corresponds to a pole at $(0, \pm)$ on $X$. Following Section \ref{ModelRH}, this gives rise to an anti-symmetric solution of the form
\begin{align} \label{singE}
    \hat E_a(z) = \frac{\theta_3(z+\frac{1}{4}-\frac{\ti \Lambda}{2}-\frac{\tau}{2}|\tau)\theta_3(z-\frac{1}{4}-\frac{\ti \Lambda}{2}+\frac{\tau}{2}|\tau)}{\theta_3(z+\frac{1}{4}|\tau)\theta_3(z-\frac{1}{4}|\tau)}, \qquad \hat N_a(p) = \hat E_a(A(p)).
\end{align}
We can now define $M^{\text{mod}}_3(k)$  analogously to $M^{\text{mod}}_1(k)$, but substituting $\hat N_a(k)$ for $N_a(k)$, and including the correct normalization at infinity for $\ti \Lambda \not = 0 \pmod 1$:
\[
M^{\text{mod}}_3(k) = \frac{\ti{\ga}(k)}{N_s(\infty_+)\hat N_a(\infty_+)} \begin{pmatrix}\alpha_3(k) & \beta_3(-k)\\ \beta_3(k) & \alpha_3(-k)\end{pmatrix}, \qquad \ti \Lambda \not = 0 \pmod 1,
\]
where
\begin{align*}
\alpha_3(k)&= \frac{1}{2} \big(\hat N_a(\infty_+) N_s(k) + N_s(\infty_+) \hat N_a(k) \big), 
\\
\beta_3(k)&= \frac{1}{2} \big( \hat N_a(\infty_+) N_s(k) - N_s(\infty_+) \hat N_a(k) \big),
\end{align*}
Note that $N_s(\infty_+)\hat N_a(\infty_+)\not = 0$ for $\ti \Lambda \not = 0 \pmod 1$ and thus we have:
\begin{align*}
    \lim_{k\to \infty} M^{\text{mod}}_3(k) = \begin{pmatrix}
    1 & 0
    \\
    0 & 1
    \end{pmatrix}, \qquad \ti \Lambda \not = 0 \pmod 1.
\end{align*}
Moreover, $\det M^{\text{mod}}_3(k)$ is an even meromorphic function with at most a simple pole at the origin, hence $\det M^{\text{mod}}_3(k) \equiv \det M^{\text{mod}}_3(\infty)= 1$. We do not define $M^{\text{mod}}_3(k)$ for $\ti \Lambda = 0 \pmod 1$, which should not pose a problem in applications, as explained in the next section.
\section{Comparison to previous work}
It turns out that sacrificing holomorphicity, while retaining \eqref{normcondM} and \eqref{detcondM}, is the most convenient way to deal with the ill-posedness of the holomorphic matrix-valued model problem for the KdV equation. 
Indeed, this was the strategy in \cite{EPT}. Note however, that the anti-symmetric meromorphic vector solutions found \cite{EPT} and \cite{GGM} are not the same as given by \eqref{singE}. The reason is, that while we assumed that $\hat N_a(p)=\hat E_a(A(p))$ has only poles at $0_\pm$, the pole condition was not necessary, as we could still allow for singularities at $\pm \I a$, as is the case for $m^{\text{mod}}(k)$. Indeed, $m^{\text{mod}}(k)/k$ is an anti-symmetric solution to the vanishing problem where solutions are required to vanish at infinity. Hence, there is no chance for uniqueness if we allow for poles at $0$ and fourth root singularities at $\pm \I a$. Moreover for $\ti \Lambda = 0 \pmod 1$, $m^{\text{mod}}(k)$ has no singularities at $\pm \I a$, and hence by our uniqueness Lemma \ref{uniqueLemma} must coincide with the solution generated by $\hat N_a(p)$. 

Interestingly, any anti-symmetric solution with a simple pole at the origin and fourth root singularities at $\pm \I a$, which is normalized to $(-1, 1)$ at infinity, is adequate for the analysis performed in \cite{EPT}. The reason is that the pole cancellation in the final step of the nonlinear steepest descent analysis is due to the underlying symmetry class, rather than the exact form of the second vector-valued solution (see Lemma 6.4 in \cite{EPT}). While the solution generated by \eqref{singE} is not normalizable for $\ti \Lambda = 0 \pmod 1$, this is not an issue, as for these values of $\ti \Lambda$ there is a regular matrix-valued model solution given in terms of \eqref{SymSol}, \eqref{AntiSol}  anyways.

\end{document}